\documentclass[11pt]{amsart}
\usepackage{amsmath}
\usepackage{amsfonts}
\usepackage{amsthm}
\usepackage{amssymb}
\usepackage{mathrsfs}
\usepackage{latexsym}
\usepackage{verbatim}
\usepackage{diagrams}
\usepackage{a4wide}
\usepackage{enumerate}
\usepackage{hyperref}

\def\cc{[\iota(c)]}
\def\VL{\widehat{V}}
\def\IA{I\kern-0.08em{A}}
\def\AN{A}

\def\Ind{\mathrm{Ind}}
\def\Afinite{\A^{\infty}}

\def\T{\mathbf{T}}
\def\Z{\mathbf{Z}}
\def\G{\mathbb{G}}
\def\A{\mathbb{A}}
\def\OL{\mathcal{O}}

\def\Q{\mathbf{Q}}
\def\F{\mathbf{F}}
\def\R{\mathbf{R}}

\def\m{\mathfrak{m}}

\def\SO{\mathrm{SO}}

\def\Ind{\mathrm{Ind}}

\def\Ext{\mathrm{Ext}}
\def\Hom{\mathrm{Hom}}

\def\GL{\mathrm{GL}}

\def\SL{\mathrm{SL}}
\def\BSL{\mathrm{BSL}}
\def\Frob{\mathrm{Frob}}

\def\PZ{\Gamma_{P}}

\newarrow{Equals}{=}{=}{=}{=}{=}
\newarrow{pots}{.}{.}{.}{.}{>}

\newtheorem{theorem}{Theorem}[section]

\newtheorem{df}[theorem]{Definition}
\newtheorem{lemma}[theorem]{Lemma}
\newtheorem{prop}[theorem]{Proposition}
\newtheorem{corr}[theorem]{Corollary}
\newtheorem{conj}[theorem]{Conjecture}
\newtheorem{example}[theorem]{Example}
\newtheorem{remark}[theorem]{Remark}

\def\Kprime{L}
\def\Htw{\widetilde{H}}

\begin{document}

\title{Hecke operators on stable cohomology}
\author{Frank Calegari \and Matthew Emerton}
\thanks{The first author was supported in part by NSF Career Grant DMS-0846285.
The second author was supported in part by NSF Grant DMS-1003339.}

\maketitle

{\scriptsize
\tableofcontents
}

\section{Introduction}

Let $\Gamma$ denote a congruence subgroup of $\GL_n(\Z)$ of level $N$.
If $p$ is prime and $\F$ is a field of characteristic $p$, then each of the cohomology groups
 $H^d(\Gamma,\F)$  admits a natural action of  a commutative ring $\T$ of Hecke operators
 $T_{\ell,k}$ for $1 \le k \le n$ and for primes $\ell$ not dividing~$Np$.
  Let $[c] \in H^d(\Gamma,\F)$
denote an eigenclass for the action of $\T$ with eigenvalues $a(\ell,k) \in \F$. 
Conjecture~B of~\cite{Ash} predicts that,
associated to $[c]$, there exists a continuous semisimple Galois 
representation $\rho: G_{\Q} \rightarrow \GL_n(\F)$ unramified outside $Np$ such that
\[\sum (-1)^k \ell^{k(k-1)/2} a(\ell,k) X^k =  \det(I  - \rho(\Frob_{\ell})^{-1} X)
\tag{$\star$}.\]
Scholze has recently announced a proof of this conjecture~\cite{Scholze}. In this note,
we  study the cohomology in low degree relative to $n$. 
Our first theorem is the following.
Let $\omega: G_{\Q} \rightarrow \F^{\times}$ denote the mod-$p$ cyclotomic character.

\begin{theorem} \label{theorem:ash} Fix an integer $d$, and suppose that $n$ is sufficiently large
compared to $d$. Then for any
 eigenclass $[c] \in H^d(\Gamma,\F)$, there exists
a character $\chi$ of conductor $N$ and a Galois representation:
$$\rho = \chi \otimes \left(1 \oplus \omega \oplus \omega^2 \oplus \ldots \oplus \omega^{n-1} \right),$$
such that $\rho(\Frob_{\ell})$ satisfies $\star$.
\end{theorem}

\begin{remark}{\em
It follows from the argument that $n \ge 2d + 6$ will suffice.
}
\end{remark}

This theorem should be interpreted as saying that the action of Hecke on \emph{stable} cohomology
is trivial, which is indeed how we shall prove this theorem.
Note that if $[c]$ comes from characteristic zero, then the result follows from a theorem
of Borel~\cite{Borel}, which shows in particular
that the only rational cohomology in low degrees arises from the trivial
automorphic representation.

\begin{example} \emph{A special case of a construction due to Soul\'{e}~\cite{Soule} implies that
$K_{22}(\Z)$ contains an element of order $691$.\footnote{Indeed, $K_{22}(\Z) \simeq \Z/691\Z$.}
Associated to $[c]$ via the Hurewicz map is a stable class $[c]$ in $H_{22}(\GL_n(\Z),\F_{691})$ for all sufficiently large $n$.
Our theorem implies that the class $[c]$ is associated to
the representation
$\rho:=1 \oplus \omega \oplus \ldots \oplus \omega^{n-1}$ via $\star$. On the other hand, the existence of $[c]$
corresponds  --- implicitly ---  to
the existence of a \emph{non-semisimple} Galois representation $\varrho$ with
$\varrho^{\mathrm{ss}} = 1 \oplus \omega^{11}$ such that the extension
class in $\Ext^1(\omega^{11},1)$ is unramified everywhere.
 Is there a generalization of Ash's conjectures 
 that predicts the existence of a \emph{non-semisimple} Galois representations associated
to Eisenstein Hecke eigenclasses?}
\end{example}

The second theorem we prove is a \emph{stability} theorem for
\emph{completed homology} (and completed cohomology). Let $\Gamma(p^k)$
denote the principal congruence subgroup of level $Np^k$.   Recall
that the completed homology and cohomology groups
are defined as follows~\cite{CE}:
$$\Htw_*(\F_p):= \varprojlim H_*(\Gamma(p^r),\F_p)  \qquad \Htw^*(\F_p) :=\varinjlim H^*(\Gamma(p^r),\F_p),$$
$$\Htw_*(\Z_p):= \varprojlim H_*(\Gamma(p^r),\Z_p)  \qquad \Htw^*(\Q_p/\Z_p) :=\varinjlim H^*(\Gamma(p^r),\Q_p/\Z_p).$$
We prove the following:

\begin{theorem} \label{theorem:MF} 
The modules $\Htw_{d}(\Z_p)$  and $\Htw_{d}(\F_p)$ stabilize as $n \rightarrow \infty$ 
and are  finite $\Z_p$-modules and $\F_p$-vector spaces respectively.  Moreover, the actions of 
$\SL_n(\Q_p)$ on both modules are trivial.
\end{theorem}

The consequences of this result will be taken up in the companion
paper~\cite{CK}).

\begin{remark}
{\em
That Theorem~\ref{theorem:ash} (or something similar)
  might be true was suggested by Akshay Venkatesh in discussions with the
first author.  (See~\cite{Ash2} for a discussion of this conjecture
and some partial results.)
}
\end{remark}


\subsection{Acknowledgements}
The first author would like to thank Akshay Venkatesh, and both authors
are grateful to
Andy Putman for answering questions about his paper~\cite{Putman}.
 We would also like to thank Konstantin Ardakov, Jean-Fran{\c{c}}ois Dat, Florian Herzig,  Akshay Venkatesh, and Simon
Wadsley for helpful remarks. 
 An earlier version of this preprint  from 2012 established
Theorem~\ref{theorem:MF} \emph{conditionally} on Conjecture~\ref{conj:AW}; the 
unconditional proof of Theorem~\ref{theorem:MF} was found in October of 2013.
The first author would especially like to thank Florian Herzig for some
remarks concerning admissible representations that led to the proof of
Theorem~\ref{theorem:MF}.

\section{Arithmetic Manifolds}

Torsion free arithmetic groups are well known to act freely and properly discontinuously on their associated
symmetric spaces. This allows one to translate questions concerning the cohomology of arithmetic groups
into questions about cohomology of the associated arithmetic quotients. While one can study such spaces
independently of any adelic framework, it is more natural from the perspective of Hecke operators to work in this
generality, and this is the approach we adopt in this paper.

\subsection{Cohomology of arithmetic quotients}
Let $K_{\infty}$ denote a fixed maximal compact subgroup of $\GL_n(\R)$, and
let $K^0_{\infty}$ denote the connected component containing the identity.
One has isomorphisms $K_{\infty} \simeq \mathrm{O}(n)$ and $K^0_{\infty} \simeq \SO(n)$.
Let $\A$ denote the  adeles of $\Q$. For any finite set of places $S$, let $\A^{S}$ denote the adeles with
the components at the places $v | S$ missing, so (for example) $\A^{\infty}$ denotes the finite adeles.
Fix a compact open subgroup $K^{\ell}$ of $\GL_n(\A^{\ell,\infty})$,  let
$K_{\ell}$ denote a compact open subgroup of $\GL_n(\Z_{\ell})$, and let $K = K^{\ell} K_{\ell}$.
Let
$$Y(K) = \GL_n(\Q) \backslash \GL_n(\A)  \slash K^0_{\infty} K^{\ell} K_{\ell}$$
denote the corresponding arithmetic quotient.

\medskip

Assume that $\F$ is a finite field of characteristic
$p \ne \ell$. We now let $\Pi_{n,\ell}$ denote the direct limit
$$ \Pi_{n,\ell} =  \varinjlim_{K_{\ell}}  H^d(Y(K),\F).$$
For $K_{\ell}$ sufficiently small, the quotient $Y(K)$ is a manifold consisting of a finite number of connected components,
all of which are $K(\pi,1)$ spaces. In particular, if $K_{\ell}$ is the level $\ell^k$ congruence subgroup,  and $K^{\ell}$ is the open subgroup of tame level $N$, 
then writing the associated space as $Y(K,\ell^k)$  there is an isomorphism
$$H^d(Y(K,\ell^k),\F) \simeq \bigoplus_{\AN} H^d(\Gamma(\ell^k),\F),$$
where $\AN:= \Q^{\times} \backslash \A^{\infty,\times}/\det(K)$ is a ray class group of conductor $N$ times a power of $\ell$, and $\Gamma$ is the corresponding classical congruence
subgroup of level $N$.

\medskip

The module $\Pi_{n,\ell}$ is endowed tautologically with an action of $\GL_n(\Q_{\ell})$ that is \emph{admissible}; that is,
letting $G(\ell^k)$ denote the full congruence subgroup of $\GL_n(\Z_{\ell})$
of level $\ell^k$, 
we have that
$$ \dim_{\F} \Pi_{n,\ell}^{G(\ell^k)} < \infty$$
for any $k$.
This property records nothing more than the finite dimensionality of the cohomology groups 
$H^d(Y(K,\ell^k),\F)$.

\section{Stability for \texorpdfstring{$\ell \ne p$}{ell noteq p}}

We use the following two key inputs to prove our result. The first is as follows:

\begin{prop} \label{prop:charney} There is an isomorphism $\Pi_{n,\ell} \rightarrow \Pi_{n+1,\ell}$ for all
$n \ge 2d + 6$.
\end{prop}

\begin{proof} For any fixed $\ell$-power congruence subgroup, this follows immediately
from the main result of Charney~\cite{Charney} (in particular, ~\S5.4 Example~(iv), p.2118). The theorem then follow by taking direct limits.
\end{proof}

\begin{lemma} \label{lemma:det} For $n \ge 2d + 6$, the action of $\GL_n(\Q_{\ell})$ on $\Pi_{n,\ell}$ is via
the determinant.
\end{lemma}

\begin{proof} By~\cite{Charney}, the action
of  $\SL_n(\Z_{\ell})$  on
$H^*(\Gamma(\ell^m),\F)$ is trivial in the stable range (see also
Lemma~7 of~\cite{Ash2}), and hence the action
of $\SL_n(\Z_{\ell})$ on $\Pi_{n,\ell}$ is also trivial.
Let $V$ be an irreducible sub-quotient of $\Pi_{n,\ell}$. 
The only admissible representations of $\GL_n(\Q_{\ell})$ that are trivial on
$\SL_n(\Z_{\ell})$ are given by characters (the normal closure of
$\SL_n(\Z_{\ell})$ inside $\SL_n(\Q_{\ell})$ is the entire group), and it follows that every irreducible constituent of 
$\Pi_{n,\ell}$ is a character. Since the action of $\GL_n(\Q_{\ell})$ on the extension of any two characters still acts through
the determinant, the result follows for $\Pi_{n,\ell}$.
\end{proof}

\begin{remark} \emph{One easy way to see that $\SL_n(\Z_{\ell})$ must act trivially on
$H^*(\Gamma(\ell^m),\F)$ in the stable range is that the latter module has fixed dimension for all $n$,
whereas the smallest non-trivial representation of $\SL_n(\Z_{\ell})$ over $\F$ has unbounded dimension as $n$ increases.}
\end{remark}

We now sketch two further proofs of Lemma~\ref{lemma:det} (one incomplete).
The first proof ``explains'' the triviality of the action of $\SL_n(\Q_{\ell})$ by showing that
$\Pi_{n,\ell}$ is very small in a precise way.

\begin{prop} \label{prop:gabber} For sufficiently large $n$, there is, for all $m$, 
an isomorphism 
$$H^d(\Gamma(\ell),\F) 
\simeq H^d(\Gamma(\ell^m),\F).$$
\end{prop}

\begin{proof} 
The main result of~\cite{Charney} identifies the stable cohomology groups $H^*(\Gamma(\ell),\F)$
as well as the groups  $H^*(\Gamma(\ell^m),\F)$ with the cohomology of the homotopy fibre of the map
$\BSL(\Z)^{+} \rightarrow \BSL(\Z/N \ell)^{+}$ and
$\BSL(\Z)^{+} \rightarrow \BSL(\Z/N \ell^m)^{+}$ respectively (under the assumption that $\F$ has
characteristic prime to $N \ell$). The natural
map between these fibre sequences induces an isomorphism on $\F$-homology, trivially for
$\BSL(\Z)^{+}$ and for $ \BSL(\Z/N \ell^m)^{+} \rightarrow \BSL(\Z/N \ell)^{+}$ by Gabber's rigidity theorem~\cite{Gabber}. 
By the Zeeman comparison theorem, this implies
the result for the homotopy fibre.
\end{proof}

More details of this argument can  be found in the companion paper~\cite{CK} 
(see Remark~1.13). Alternatively, one can prove the result directly in this case, since the transfer
map from $\Gamma(\ell^m)$ to $\Gamma(\ell)$ is an isomorphism over any ring such that  $\ell$ is inverted (the group $\Gamma(\ell)/\Gamma(\ell^m)$ is an $\ell$-group), and hence it suffices to show that the $\SL_n(\Z_{\ell})$ action on
$H^d(\Gamma(\ell^m),\F)$ is trivial.
 Lemma~\ref{lemma:det} is an easy consequence of Proposition~\ref{prop:gabber}. 
For our next argument, we begin with the following:

\begin{prop}  \label{prop:GK} Let $\VL$ be an irreducible admissible infinite dimensional
representation of $\GL_n(\Q_{\ell})$ in characteristic zero.
Then the 
  Gelfand--Kirillov dimension of $\VL$ is at least $n-1$.
  \end{prop}
  
\begin{proof}  The Gelfand--Kirillov dimension of $\VL$ can be interpreted in two ways.
On the one hand, it is that value of $d\geq 0$ for which $\dim V^{G(\ell^k)} \asymp \ell^{dk}$;
on the other hand, if we pull back the 
 Harish--Chandra character 
to a neighbourhood of $0$ in $\mathfrak g$, the Lie algebra of $\GL_n(\Q_{\ell})$,
via the exponential map, then a result of Howe and Harish--Chandra shows that
the resulting function $\chi$ admits an expansion
$$\chi = \sum_{\mathcal O} c_{\mathcal O} \widehat{\mu}_{\mathcal O},$$
where the sum ranges over the nilpotent $\GL_n(\Q_{\ell})$-orbits in $\mathfrak g^{\vee}$,
and $\widehat{\mu}_{\mathcal O}$ denotes the Fourier transform of a suitably normalized
$\GL_n(\Q_{\ell})$-invariant measure on $\mathcal O$ (see e.g.\ the introduction
of \cite{GS}, whose notation we are following here,
for a discussion of these ideas); the Gelfand--Kirillov dimension
of $\VL$ is then also equal $\dfrac{1}{2}\max_{\mathcal O \, | \, c_{\mathcal O} \neq 0}
\dim \mathcal O$ (as one sees by pairing the characteristic function
of $G(\ell^k)$ --- pulled back to $\mathfrak g$ --- against $\chi$). 
Since $\VL$ has Gelfand--Kirillov dimension $0$ if and 
only if it is finite dimensional, and since the minimal non-zero nilpotent
orbit in $\mathfrak g$ has dimension $2(n - 1)$
(see~\S1.3 p.~459 and Table~1 on p.~460 of~\cite{MR698349}),
the proposition follows.
\end{proof}

Let $V$ be an irreducible sub-quotient of $\Pi_{n,\ell}$. Since the cohomology of pro-$\ell$ groups vanishes
in characteristic $p \ne \ell$, we deduce that, for all $k$,
$$\dim V^{G(\ell^k)} \le \dim \Pi^{G(\ell^k)}_{n,\ell} 
= \dim \Pi^{G(\ell^k)}_{m,\ell},$$
where $m = 2d + 6$ is fixed, and where the equality follows from Propostion~\ref{prop:charney}.
Yet, for a fixed $m$, there is the trivial inequality relating the growth of cohomology to the growth of the index
(up to a constant) and thus
$$ \dim \Pi^{G(\ell^k)}_{m,\ell} \ll [\GL_m(\Z_{\ell}):G(\ell^k)] \ll \ell^{k m^2}.$$
 The bound on invariant growth
then implies that the  Gelfand--Kirillov dimension of $\VL$ is at most $m^2$.
Suppose that $V$ lifted to a (necessarily irreducible) representation $\VL$ in characteristic $0$. We deduce a corresponding bound for the invariants of $V$.
By Proposition~\ref{prop:GK},
this is a contradiction for sufficiently large $n$
unless $\VL$  and thus $V$ is one-dimensional. 
In general,  it follows from the main theorem of Vign{\'e}ras (\cite{Vigneras}, p.182) that any irreducible
admissible irreducible representation $V$ of $\GL_n(\Q_{\ell})$ over $\F$ lifts 
to a  \emph{virtual} representation $[\VL]$ in characteristic $0$.  We \emph{expect} that $[\VL]$
may be chosen to contain a unique irreducible representation of largest Gelfand--Kirillov dimension.
If this is so, then the bound of Proposition~\ref{prop:GK} applies also to mod-$p$ representations of
$\GL_n(\Q_{\ell})$, and it would follow that every irreducible constituent of 
$\Pi_{n,\ell}$ is a character. Since the action of $\GL_n(\Q_{\ell})$ on the extension of any two characters still acts through
the determinant, this would yield a third proof of Proposition~\ref{lemma:det}.

\begin{remark}  \emph{In a preliminary version of this paper, we made the error of assuming that any irreducible 
$\GL_n(\Q_{\ell})$-representation $V$ in characteristic $p$ lifted to a characteristic zero representation $\VL$.
We thank Jean-Fran{\c{c}}ois Dat for pointing out to us that this is not
always the case, and also for suggesting a possible approach to proving
our modified expectation via induction on the ordering of partitions associated to $V$.  Because of the availability of other arguments, however, we have not made a serious attempt to carry out
this one.
Although this  alternate argument is therefore  incomplete,
it is the one that is most amenable to generalization to the case of
$\ell = p$ (see~\S\ref{section:lp}),
and so for psychological reasons we have presented it here.}
\end{remark}

\section{Hecke Operators}
Let $g \in \GL_n(\A^{\infty})$ be invertible. Associated to $g$ one has the Hecke operator $T(g)$,
 defined by considering the composition:
$$H^{\bullet}(Y(K),\F) \rightarrow
H^{\bullet}(Y(g K g^{-1} \cap K),\F) \rightarrow H^{\bullet}(Y(K  \cap g^{-1} K g),\F) \rightarrow
H^{\bullet}(Y(K),\F),$$
the first map coming from the obvious inclusion, the final coming from
corestriction map. The Hecke operators preserve $H^{\bullet}(Y(K),\F)$, but
not necessarily
the cohomology of the connected components. Indeed, the action on the component
group is via the determinant map on $\G(\Afinite)$ and the natural action of
$\A^{\infty,\times}$ on $\AN$.
The Hecke operator $T_{\ell,k}$ is defined by taking $g$ to be
the diagonal matrix consisting of $k$ copies of $\ell$ and $n - k$ copies of $1$.
The algebra $\T$ of endomorphisms generated by $T_{\ell,k}$ on $H^*(Y(K),\F)$ for $\ell$ prime to the level of $K$
generates a commutative algebra. For any such $T=T(g)$, let $\langle T \rangle$ denote the isomorphism of $H^*(Y(K),\F)$
that acts by permuting the components according to the image of $\det(g)$ in~$A$.

\medskip

There are many definitions of the notion of ``Eisenstein'' in many different contexts. For our purposes, the following very
restrictive definition is appropriate:

\begin{df} A cohomology class $[c] \in H^d(Y(K),\F)$ is  Eisenstein if $T  [c] = \langle T \rangle \deg(T) [c]$ for any Hecke operator $T$.
A maximal ideal $\m$ of $\T$ is Eisenstein if and only if $\m$ contains  $T  -  \deg(T)$ for all $T$ with
$\langle T \rangle = 1$ in $\AN$.
\end{df}

If $[c] \in H^d(\Gamma,\F)$ is a Hecke eigenclass such that $T[c] = \deg(T) [c]$ for all $\langle T \rangle = 1$, then
$[c]$ is necessarily Eisenstein, and moreover $T [c] = \chi(\langle T \rangle) \deg(T)[c]$ for some character $\chi:\AN \rightarrow \F^{\times}$
of $\AN$. 
By class field theory, $\chi$ corresponds to a finite order character of $G_{\Q}$ of conductor dividing $N$. 
A easy computation of $\deg(T_{\ell,k})$ then implies that $\star$ will be satisfied  with
$\rho$ equals  $\chi \otimes (1 \oplus \omega \oplus \ldots \oplus \omega^{n-1})$ if and only if $[c]$ is Eisenstein.

\subsection{Proof of Theorem~\ref{theorem:ash}}
Given any eigenclass $[c] \in H^d(Y(K),\F)$, consider its image $\cc$ in $\Pi_{n,\ell}$.
If this image is non-zero,
we can determine the eigenvalues of $[c]$ by determining those of
$\cc$.
Since the $\GL_n(\Q_{\ell})$-action on $\Pi_{n,\ell}$ factors through $\det$,
the action of any $g \in \GL_n(\Q_{\ell})$
on $\cc$, as well as the action of the $\deg(T(g))$ representatives
in the double coset decomposition of $g$, is via~$\langle T \rangle$;
hence
$T \cc - \langle T \rangle \deg(T) \cc$ is zero in $\Pi_{n,\ell}$.
It follows that either $[c]$ is Eisenstein or it lies
in the kernel of the map $H^d(Y(K),\F) \rightarrow \Pi_{n,\ell}$. Hence Theorem~\ref{theorem:ash} follows
by induction from the following Lemma.

\begin{lemma} \label{lemma:cong} Let $\m$ be a maximal ideal of $\T$ that is not Eisenstein, and suppose that
$H^i(Y(L),\F)_{\m} = 0$ for all $i < d$ and $L = K^{\ell} L_{\ell}$
for all compact normal open subgroups $L_{\ell} \subset \GL_n(\Z_{\ell})$.
If $[c]$ lies in the kernel of the map $H^d(Y(K),\F) \rightarrow \Pi_{n,\ell}$, then $[c]$ is Eisenstein.
\end{lemma}

 \begin{proof} By assumption, the class $[c]$ lies in the kernel of the map
 $$H^d(Y(K),\F) \rightarrow H^d(Y(L),\F)$$
 for some $L$ as above. It suffices to show that this kernel vanishes after completion at any
 non-Eisenstein prime $\m$. 
  By assumption, the cohomology of $Y(K)$ localized at $\m$ vanishes in degree less than $d$. 
  The localization of 
  the Hochschild--Serre spectral sequence at $\m$ thus becomes an inflation-restriction sequence,
  from which we deduce that
    the set of classes in $H^d(Y(K),\F)_{\m}$ that are annihilated under the level raising map  is isomorphic
   to a finite number of copies of the group $H^d(K/L,\F)_{\m}$.
   Let $g \in \GL_n(\A)$. There is a canonical isomorphism
$$(g K g^{-1} \cap K)/(g \Kprime g^{-1} \cap \Kprime) \simeq K/\Kprime.$$
Hence we obtain a commutative diagram:
  $$
 \begin{diagram}
 H^d(K/\Kprime,\F) & \rTo & H^d(K/\Kprime,\F) &   \rTo & H^d(K/\Kprime,\F) &  \rTo^{\phi \quad} & H^d(K/\Kprime,\F) \\
 \dTo & & \dTo & & \dTo & & \dTo \\
 H^{d}(Y(K),\F) & \rTo &  
H^{d}(Y(g K g^{-1} \cap K),\F) & \rTo & H^{d}(Y(K  \cap g^{-1} \cap K g),\F) & \rTo  &
H^{d}(Y(K),\F), \\
\end{diagram}
$$
It follows that the action of $T = T(g)$ is given by multiplication by $\det(T)$ composed with the permutation 
$\langle T \rangle = \det(g) \in \AN$ of components. This action is Eisenstein, and  hence the kernel vanishes for any non-Eisenstein
prime $\m$.
   \end{proof}

This completes the proof of Theorem~\ref{theorem:ash}.

\begin{remark} \emph{By the universal coefficient theorem,
one has 
$$\Pi^{\vee}_{n,\ell}:=\Hom(\Pi_{n,\ell},\F) = \varprojlim H_d(Y(K),\F).$$
 Suppose that one defines 
 $$\Pi^{\vee}_{n,\ell}(\Z_{\ell}) = \varprojlim H_d(Y(K),\Z_{\ell}).$$
 Then it follows from
Lemma~\ref{lemma:det} for $d$ and $d+1$ that the action of $\GL_n(\Q_{\ell})$ on
$\Pi^{\vee}_{n,\ell}(\Z_{\ell})$  is via the determinant for sufficiently large $n$.}
\end{remark}

 \begin{remark} \emph{One can ask whether a Hecke operator $T$ with
 $\langle T \rangle$ trivial in $\AN$ acts via the degree on the entire cohomology group
 $H^d(\Gamma,\Z)$. Our argument shows that, for such $T$, 
 the image of $T - \deg(T)$ on $H^d(\Gamma,\Z)$ is --- in the notation of~\cite{CV} --- \emph{congruence}; i.e.,
 it lies in the kernel of the map $H^d(\Gamma,\Z) \rightarrow H^d(\Gamma(M),\Z)$
 for some $M$.
 }
 \end{remark}
 
 \begin{remark} \emph{The main theorem and its proof remain valid, and essentially unchanged, if one
replaces $\GL_n(\Z)$ by $\GL_n(\OL_F)$ for any number field $F$. }
\end{remark}

 \section{Stability for \texorpdfstring{$\ell = p$}{ell = p}}
 \label{section:lp}
It is natural to wonder whether
 the methods of this paper can be extended to $\ell = p$.
One obvious obstruction is that the na\"{\i}ve notion of stability fails, even for $d = 1$:
by the congruence subgroup property~\cite{BLS,Mennicke},
and letting 
$\SL(n,p^k)$ denote the principal congruence subgroup 
of $\SL_n(\Z)$ of level $p^k$,  
one easily sees that, for $n \ge 3$,
$$H_1(\Gamma(p),\Z) \simeq \Gamma(p)/\SL(n,p^2) \simeq  (\Z/p \Z)^{n^2 - 1},$$
which clearly do \emph{not} stabilize.
 On the other hand, the non-trivial classes in these groups arise
 for purely local reasons, namely, 
as the pullback of classes from the homology of the $p$-congruence subgroup of the congruence completion
$\SL_ n(\Z_p)$. This latter homology is  evidently
insensitive to the finer arithmetic properties of $\SL_n(\Z)$,
and so it is natural to excise the cohomology arising
for ``local'' reasons and consider what remains.
We do this via the method of completed cohomlogy,
as discussed in \cite{CEAnn,CE,Emerton}, whose definition
for a finite field $\F$ we recall below.
It is completed cohomology that will provide
the analogue of the modules $\Pi_{n,\ell}$ defined above.

We fix, once and for all, a tame level $N$,
and let $\Gamma(p^k)$ denote the principal congruence subgroup 
of $\SL_n(\Z)$ of level $Np^k$.  

\begin{df} The completed cohomology groups $\Htw^{d}_n$ are defined as follows~\cite{CE}:
$$\Htw^{d}_n := \varinjlim H^d(\Gamma(p^k),\F)$$
\end{df}

Although this definition is formally the same as $\Pi_{n,\ell}$, the theory when
$\ell = p$ is quite different to the $\ell \neq p$ case,
due to the non-semisimple nature of representations of pro-$p$ groups
on mod-$p$ vector spaces.
However, we still prove the following result:

\begin{theorem}[Stability of completed cohomology] \label{theorem:stable} For $n$ sufficiently large, the modules $\Htw^{d}_n$ are finite
dimensional over $\F$
 and are  independent of $n$.
 \end{theorem}

\begin{remark} \emph{We have an isomorphism
$\displaystyle{\Htw_{d,n}:=\Hom(\Htw^{d}_{n},\F) \simeq \varprojlim H_d(\Gamma(p^k),\F)}$.
If we define
$$\displaystyle{\Htw_{d,n}(\Z_p):= \varprojlim H_d(\Gamma(p^k),\Z_p)},$$
then Theorem~\ref{theorem:stable} also implies, by Nakayama's lemma, that $\Htw_{d,n}(\Z_p)$ is a finite $\Z_p$-module for sufficiently large $n$.
Hence Theorem~\ref{theorem:stable}  implies Thoerem~\ref{theorem:MF}.
}
\end{remark}

\begin{remark} 
\emph{ Theorem~\ref{theorem:stable} is immediate
when $d = 0$ and $d = 1$. Indeed, for $d = 0$ one 
has $\Htw^0_n(\F) = \F$ for all~$n$, while for $d = 1$
one has $\Htw^1_n(\F) = 0$ for all $n \ge 3$ by the congruence subgroup property.
}
\end{remark}

\medskip

The proof of Theorem~\ref{theorem:stable} will occupy most of the remainder of the paper.
We begin by recalling some facts concerning non-commutative Iwasawa theory.
 Let  $K = \SL_n(\Z_p)$ 
 and let $K(p^k)$ denote the principal congruence subgroups of $K$.
By construction, the module $\Htw^d_n$ is naturally a module over the completed group
ring $\Lambda = \Lambda_{\F_p}:=\F_p[[K(p)]]$. If $M$ is a $\Lambda$-module, then let $M^{\vee}:=\Hom(M,\F_p)$
be the dual $\Lambda$-module. 
By Nakyama's Lemma, $\Htw_{d,n}$ is finitely generated, which implies (by definition)
 that $\Htw^{d}_n$ is co-finitely generated.
 The ring $\Lambda$ is Auslander regular~\cite{Venjakob}, which implies that there is
a nice notion of dimension and co-dimension of  finitely generated $\Lambda$-modules.
One characterization of dimension for modules is given by the 
following result
 of Ardakov and Brown~\cite{MR2290583}:

\begin{prop} \label{prop:boundtwo}
If $M$ is finitely generated  and $M^{\vee}$ is the co-finitely generated dual, then 
 $M$ has
dimension at most $m$ if and only if, as $k$ increases without bound,
\[ \dim (M^{\vee})^{K(p^k)} \ll p^{mk}. \] 
\end{prop}

The following is the natural analogue of Proposition~\ref{prop:GK} in this context.

\begin{conj} \label{conj:AW} 
Let $M$ be a finitely generated $\Z_p[[K(p)]]$-module. If $M$ is infinite dimensional
over $\F$, then the dimension of $M$ is at least~$n - 1$.
\end{conj}

\begin{remark} \emph{The analogue of this conjecture for $\Lambda_{\Q_p}$ is true
by  Theorem~A of~\cite{Ardakov}. In particular, if  any finitely generated pure $\Lambda_{\F_p}$-module $M$
is the reduction of
of a $p$-torsion free $\Lambda_{\Z_p} = \Z_p[[K(p)]]$-module, then the conjecture is true. Such a lifting always
exists for \emph{commutative} regular local rings, but it is unclear whether one should expect it to hold
for $\Lambda$.}
\end{remark}

Our arguments  proceed in a manner quite similar to the $\ell \ne p$ case.
One missing ingredient is that
Proposition~\ref{prop:charney}  is no longer valid.  We use the
central stability results of Putman~\cite{Putman} as a replacement.

Let $(V_n)$ be a collection of representations of $S_n$.
Recall from~\cite{Putman} that a sequence
$$\begin{diagram}
 \ldots & \rTo & V_{n-1}  & \rTo_{\phi_{n-1}} & V_n  & \rTo_{\phi_n} & V_{n+1}   & \rTo & \ldots  \end{diagram}$$
 is \emph{centrally stable} if:
 \begin{enumerate}
 \item The $\phi_*$ are equivariant with respect to the natural inclusions on symmetric groups,
 \item For each $n$,  the morphism
$\Ind_{S_n}^{S_{n+1}} V_n \to V_{n+1}$ induced by $\phi_n$ identifies $V_{n+1}$ 
 with the (unique) maximal quotient of
 $\mathrm{Ind}^{S_{n+1}}_{S_n} V_n$ such that the $2$-cycle $(n,n+1)$ acts trivially on $\phi_{n-1}(V_{n-1})$.
 \end{enumerate}

\begin{lemma} \label{lemma:putman} If Theorem~{\em \ref{theorem:stable}} holds for $j < d$,
then the modules $\Htw_{d,n}$ are centrally stable for sufficiently large $n$.
\end{lemma}

\begin{proof} This argument is essentially taken directly from Putman~\cite{Putman}, 
and we follow his argument closely.
Putman considers a spectral sequence at a fixed level, which we may take to be
$\Gamma(p^k)$.
Taking the inverse limit, 
one obtains a corresponding spectral sequence for completed homology. In order for this sequence
to degenerate at the relevant terms on page~$2$, it suffices to show that the appropriate $E^{2}_{i,j}$
are actually zero. By assumption, the modules $\Htw_{j,n}$ for $j < d$ are finite dimensional
vector spaces that are independent of $n$ and have a trivial $S_n$-action. Thus it suffices 
(following Putman) to show that the $(n+1)$st central stability  complex for the 
trivial sequence:
$$\F \rightarrow \F \rightarrow \F \rightarrow \ldots$$
of modules for the symmetric group (starting at the trivial group) is exact.
This is a very special case of Proposition~6.1 of~\cite{Putman}, but
can be verified directly in this case. Hence one deduces --- as in Putman --- that
the $\Htw_{d,n}$ are centrally stable.
\end{proof}

\begin{remark} \emph{It is expected that forthcoming work of Church and Ellenberg will
provide a direct proof that  the groups $H_d(\Gamma(p^k),\F)$ are centrally stable. After taking
inverse limits in~$d$, this would imply that the $\Htw_d$ are centrally stable without any induction
hypothesis and
provide another (although not entirely unrelated) proof of Lemma~\ref{lemma:putman}.
}
\end{remark}

\begin{corr} \label{corr:corr} If Theorem~{\em \ref{theorem:stable}} holds for $j < d$,
then the dimension of $\Htw_{d,n}$  as a  $\Lambda_{\F_p}$-module is bounded independently
of $n$.
\end{corr}

\begin{proof} This is obvious for any fixed collection of $n$. Yet, for $n$ sufficiently large, 
central stability implies that the natural map
$$\Ind^{S_{n+1}}_{S_n} \Htw_{d,n} \rightarrow \Htw_{d,n+1}$$
induced by the $n+1$ embeddings of $\SL(n)$ into $\SL(n+1)$ is surjective.
Passing to cohomology, the natural maps $\Htw^{d}_{n+1} \rightarrow \Htw^{d}_n$ induce
homomorphisms from the $\Gamma_{n+1}(p^k)$-invariants
of $\Htw^{d}_{n+1}$ to the $\Gamma_n(p^k)$-invariants of $\Htw^{d}_n$. In particular, the
dimension of the $\Gamma_{n+1}(p^k)$-invariants of $\Htw^{d}_{n+1}$  is at most $n$ times the dimension of the 
$\Gamma_n(p^k)$-invariants of $\Htw^{d}_n$, and so it
 follows by 
Proposition~\ref{prop:boundtwo} that 
$\Htw_{d,n+1}$ has Iwasawa dimension
bounded by the Iwasawa dimension of~$\Htw_{d,n}$. 
\end{proof}

\begin{remark} \emph{Using Putman's spectral sequence, one may prove unconditionally 
that 
$$\dim_{\F} H^d(\Gamma(p^k),\F) \ll p^{m k}$$
for some constant $m$ that does not depend on $n$ (although the implied constant does depend on $n$).
This leads to an alternate proof of Corollary~\ref{corr:corr} via the Hochschild--Serre spectral sequence
and Proposition~\ref{prop:boundtwo}.
}
\end{remark}

We prove Theorem~\ref{theorem:stable} by induction.
By the proceeding corollary, 
the dimension of $\Htw_{d,n}$ must be bounded for sufficiently large $n$.
If we assume Conjecture~\ref{conj:AW}, then we immediately deduce that this
dimension is actually zero for sufficiently large $n$, and hence
$\Htw_{d,n}$ is finite and
 the action of $\SL_n(\Q_p)$ on $\Htw_{d,n}$ is trivial. In particular, the action of $S_n$
 must also be either trivial, or via the sign character, for sufficiently large $n$.
 By Lemma~\ref{lemma:putman}, we also deduce
 that the sequence $\Htw_{d,n}$ is centrally stable.
 This rules out the possibility that $S_n$ acts via the sign character (for large
 enough $n$), since these do not fit into a centrally stable sequence.
 It follows that we must have a centrally stable sequence of trivial $S_n$-representations.
Yet a centrally stable sequence of trivial modules is stable, and hence $\Htw_{d,n}$ stabilizes.

How may we upgrade this to an unconditional proof?
One thing to observe is that the modules $\Htw_{d,n}$ have an action of $\SL_n(\Q_p)$ that has not been
exploited. For technical reasons, it is actually more convenient to have an action of the group $\GL_n(\Q_p)$.  In order to upgrade
the action of $\SL_n(\Q_p)$ to $\GL_n(\Q_p)$, one should take the direct limit (in cohomology) as follows:
$$ \varinjlim_{K_{p}}  H^d(Y(K),\F).$$
From now on, we use $\Htw^d_n$ to denote this limit
(rather than the limit over the cohomology of congruence subgroups that it denoted up till now).

With this new definition,
$\Htw^d_n$ is cofinitely generated over $\F_p[[\GL_n(\Z_p)]]$ rather than $\Lambda_{\F_p} = \F_p[[\SL_n(\Z_p)]]$.
  It then suffices to show that these completions are eventually stable
and finite over the ring $\F_p[[\det(\GL_n(\Z_p))]] = \F_p[[\Z^{\times}_p]]$. 
Instead of appealing to a general conjecture concerning
finitely generated $\Lambda_{\F_p}$-modules, we only need consider such modules whose dual admits a smooth
admissible action of $\GL_n(\Q_p)$. In particular,
 it suffices to show that the only irreducible
$\GL_n(\Q_p)$-representations that occur as sub-quotients of these completions are either trivial or
have large dimension. So, instead of proving Conjecture~\ref{conj:AW}, it suffices to prove the following:

\begin{conj}
\label{conj:new} Let $\pi$ be an irreducible 
infinite-dimensional smooth admissible representation
of $\GL_n(\Q_p)$, and let $\pi^{\vee} = \Hom(\pi,\F_p)$. Then the
dimension of $\pi^{\vee}$  
 as a
$\Lambda_{\F_p}$-module is~$\ge n-1$. 
\end{conj}

Although this conjecture is presumably easier than Conjecture~\ref{conj:AW}, we still do not know how to 
prove it. Instead, we prove the following two lemmas:

\begin{lemma} \label{lemma:induction}
Let $\pi$ be an irreducible 
infinite-dimensional  smooth admissible representation of $\GL_n(\Q_p)$,
and let $\pi^{\vee} = \Hom(\pi,\F_p)$. If $\pi$
is not a supercuspidal representation, then the
dimension of $\pi^{\vee}$  
 as a
$\Lambda_{\F_p}$-module is~$\ge n-1$. 
\end{lemma}


\begin{lemma} \label{lemma:induce}
Any
 infinite-dimensional irreducible smooth admissible
representation $\pi$ that occurs as a sub-quotient of $\Htw^{d}_{n+1}$ 
 is not supercuspidal.
\end{lemma}

Taken together, these lemmas serve as a replacement for either Conjecture~\ref{conj:new}
or Conjecture~\ref{conj:AW}, and imply the main theorem.
To recapitulate  the argument, from Lemma~\ref{lemma:induce}, we deduce that
every irreducible constitutent of $\Htw^{d}_{n+1}$ is not supercuspidal.
From Lemma~\ref{lemma:induction}, we deduce that either $\Htw^{d}_{n+1}$
contains a representation of dimension at least $n$, which contradicts the uniformly bounded
dimension of $\Htw^{d}_{n+1}$ for all $n$, or the only
irreducible constituents of  $\Htw^{d}_{n+1}$ are finite, in which case the action
of $\SL_n(\Q_p)$ is also trivial.

\begin{proof}[Proof of Lemma~\ref{lemma:induction}]  
Let $G$ denote $\GL_n(\Q_p)$ and let $B$ denote a Borel of $G$. The Lemma is easy to verify directly for
the Steinberg representation, for which the Iwasawa dimension of the dual is  equal to $\dim(G/B)$.
By the main theorem of~\cite{Herzig}, we may therefore assume that  there exists a proper
parabolic $P$ with Levi $M$ such that:
$$\pi = \Ind^{G}_{P}(\sigma_1 \otimes \ldots \otimes \sigma_r) = \Ind^{G}_{P}(\sigma).$$

The dual of such an induced representation has dimension equal to the sum of the dimension
of the dual of $\sigma$ and the dimension of $G/P$, and hence
has dimension at least equal to that of $G/P$.  It is easy enough to prove
this directly, but we instead give a proof here using the chacacterization of dimension
coming from Proposition~\ref{prop:boundtwo} (which applies just as well to $\GL_n(\Z_p)$
as it does to $\SL_n(\Z_p)$).

Using the Iwasawa decomposition $G = KP$ (where now $K$ denotes $\GL_n(\Z_p)$),
we see that we may
rewrite $\pi$ (thought of as a $K$-representation) as
$\pi = \Ind_{P \cap K}^K (\sigma)$. A short calculation then shows, since the level $p^k$ congruence subgroup $K(p^k)$ is normal in $K$, that
$$\pi^{K(p^k)} = \Ind_{P(\Z/p^k)}^{G(\Z/p^k)} (\sigma^{M \cap K(p^k)}).$$
Now for $k > 0$, the representation inside the induction is nonzero. Hence 
$$\dim \pi^{K(p^k)} \ge [G(\Z/p^k):P(\Z/p^k)] \sim p^{mk},$$
  where $m = \dim(G/P)$. As claimed,
 this gives a lower bound of $\dim(G/P)$ on the Iwasawa dimension 
 of~$\pi$.

 The minimum value of
$\dim(G/P)$ amongst all proper parabolics is~$n-1$, and so the lemma follows.
\end{proof}

\begin{proof}[Proof of Lemma~\ref{lemma:induce}]  
 Let $M \subset \GL_{n+1}(\Q_p)$ denote the Levi subgroup $\Q^{\times}_p \times \GL_{n}(\Q_p)$ of
 $\GL_{n+1}(\Q_p)$,
and let $P$ be the corresponding parabolic. Associated to $P$ is a natural arithmetic group  $\PZ \subset \GL_{n+1}(\Z)$.
Let  $\PZ(p^k) = \PZ \cap \Gamma_{n+1}(p^k)$. The group $\PZ(p^k)$ is a 
semi-direct product:
$$\PZ(p^k) \simeq (p^k \Z)^{n} \rtimes \Gamma_{n}(p^k).$$
Let us compute the completed cohomology groups of $\PZ$.
There is a Hochschild--Serre spectral sequence:
$$H^i(\Gamma_{n}(p^k),H^j((p^k \Z)^{n},\F_p)) \Rightarrow H^{i+j}(\PZ(p^k),\F_p).$$
There is an isomorphism
$H^j(p^n \Z,\F_p) = \wedge^j (\F_p)^{n}$,
but the composition maps from $\PZ(p^k)$ to $\PZ(p^{k+1})$ induce the zero map on these cohomology groups
unless $j = 0$. Hence, taking direct limits, we get an isomorphism:
$$ \Htw^d_{P} :=\varinjlim H^d(\PZ(p^n),\F_p) \simeq \Htw^d_{n}.$$
Omit the degree $d$ from the notation.
We therefore have maps as follows:
$$\Htw_{n+1} \rightarrow \Htw_{P} \rightarrow \Htw_{n}.$$
The first map is $P$-equivariant, the second is an isomorphism, and the composition is 
the natural map. 
It follows that there exists a $G$-equivariant map
$$\Htw_{n+1} \rightarrow \Ind^{G}_{P} \Htw_{n}.$$
Note that $S_{n+1} \cap P = S_{n}$, and hence the right hand side includes the induction from
$\Htw_{n}$ from $S_{n}$ to $S_{n+1}$. Hence, by central stability, we actually have an injection:
$$\Htw_{n+1} \rightarrow \Ind^{G}_{P} \Htw_{n}.$$
To complete the proof of the lemma, it suffices to show
that no constituent of the right hand side is supercuspidal. But this also follows immediately from
Herzig's classification of irreducible admissible representations of $G$~\cite{Herzig},
since $\Htw_{n}$ is a smooth
admissible $M$-representation, where $M$ is the Levi of $P$.
\end{proof}

\subsection{Consequences for classical cohomology groups}
The homology (or cohomology) of arithmetic groups with $\F$ coefficients
can be recovered from completed homology via
the Hochschild--Serre spectral sequence \cite{CE}.
For example,
if we take our arithmetic group to be the principal congruence subgroup
$\Gamma(p)$ of $\SL_n(\Z)$, then, 
recalling that $G(p)$ is the principal subgroup of $\SL_n(\Z_p)$,
this spectral sequence has the form
$$E_2^{i,j} := H^i\bigl(G(p),\Htw^j\bigr) \implies H^{i+j}\bigl(\Gamma(p),
\F\bigr).$$
As noted above, $\Htw^0 = \F$, and 
Theorem~\ref{theorem:stable} implies that each $\Htw^j$ is finite
with trivial
$G(p)$-action.
Hence Theorem~\ref{theorem:stable} implies that
 the cohomology group $H^d(\Gamma(p),\F)$ has a filtration consisting of three types of classes:
those arising  via pullback from $H^d(G(p),\F)$, those arising via (higher) transgressions from
$H^j$ of $G(p)$ for $j < d$, and those arising from
$\Htw^d$ (which do not depend on $n$).
In this optic, the notion of \emph{representation stability}
developed by Church and Farb~\cite{CF} is then seen (in this context) to have its origin
in the mod-$p$ cohomology of $p$-adic Lie
groups, rather than in the properties of arithmetic groups. 

One
interpretation of Theorem~\ref{theorem:stable} is that the $\ell = p$ and $\ell \ne  p$ theories after
completion are quite similar. The difference in phenomenology  between these two cases
is then a consequence of the vanishing
of $H^i(G(\ell^k),\F)$ when $k \ge 1$ and the chararacteristic $p$
of $\F$ is not $\ell$. Another way to state these results is to say
that the only  irreducible $\GL_n(\A^{\infty})$-representations
occurring inside cohomology in sufficiently small degree are one-dimensional.

\subsection{Relation to \texorpdfstring{$K$}{K}-theory}

Even knowing that $\Htw_{d}$ is finite,
it is not at all apparent how to compute it.
One might wonder whether there is some interpretation of $\Htw_{d}$ in terms of $K$-theory
(some related speculation along these lines occurs in~\S8.3 of~\cite{CV}).
This question is taken up in the companion
paper~\cite{CK}.

\bibliographystyle{amsalpha}
\bibliography{Stable}
 
 \end{document}